\documentclass[11pt,a4paper]{article}
\usepackage{geometry}
\usepackage{amssymb}
\usepackage{amsmath}
\usepackage{amsthm}
\usepackage{subcaption}
\usepackage{tikz}

\theoremstyle{plain}
\newtheorem{theorem}{Theorem}
\newtheorem{lemma}[theorem]{Lemma}
\newtheorem{corollary}[theorem]{Corollary}

\theoremstyle{definition}

\newtheorem{example}[theorem]{Example}
\usepackage{authblk}

\title{Undirected edge geography games on grids}

\author[1]{Tharit Sereekiatdilok}

\author[1,2]{Panupong Vichitkunakorn\thanks{panupong.v@psu.ac.th}}

\affil[1]{
    Division of Computational Science, Faculty of Science, Prince of Songkla University, Songkhla, Thailand}

\affil[2]{
    Research Center in Mathematics and Statistics with Applications, Prince of Songkla University, Songkhla, Thailand}

\date{}
\begin{document}
\maketitle

\begin{abstract}

The undirected edge geography is a two-player combinatorial game on an undirected rooted graph.
The players alternatively perform a move consisting of choosing an edge incident to the root vertex, removing the chosen edge, and marking the other endpoint as a new root vertex. 
The first player who cannot perform a move is the loser.
In this paper, we are interested in the undirected edge geography game on the grid graph $P_m\square P_n$.
We completely determine all N-positions (the root vertices where the first player wins) and all P-positions (the root vertices where the first player loses). Moreover, we give a winning strategy for the winner.
\end{abstract}

\section{Introduction} \label{thm:fsu93}

In 1993, Aviezri S. Fraenkel and Shai Simonson introduced the geography game on a directed graph \cite{FRAENKEL1993197}. 
It is a two-player combinatorial game on a directed graph $G$ that starts with a token on the vertex $v$.
The players alternatively move the token to another vertex along an arrow (a directed edge) and then remove the arrow.
The first player who has no legal move is the loser.
Later in the year, the game is extended to an undirected graph and is called the \emph{Undirected Edge Geography (UEG)} \cite{Fraenkel1993}.

Let $G_{m\times n} = P_m \square P_n$ be a grid graph with $m$ columns and $n$ rows.
Let $[n] = \{1,2,\dots,n\}$. Throughout the paper, we embed $G_{m\times n}$ inside a rectangle $[0,m+1]\times[0,n+1]$ such that $V(G_{m\times n}) = [m]\times [n]$ and two vertices $(i,j)$ and $(i',j')$ are adjacent if and only if $|i-i'|+|j-j'|=1$.

We are interested in the (rooted) Undirected Edge Geography (UEG) on $G_{m\times n}$. A position (or a state) of the game can be written in the form $(H,v)$ where $H$ is a subgraph of $G_{m\times n}$ and $v\in V(G_{m\times n})$ is called the \emph{root} (the vertex the game is currently playing on).
A position is an \emph{N-position} if the first player wins, and it is a \emph{P-position} if the first player loses.
\begin{theorem}[\cite{Fraenkel1993}]
    Let $m,n\geq 2$ and $v=(1,1)$. Then $(G_{m\times n}, v)$ is a P-position if and only if $\gcd(m+1, n+1) \neq 1$.
\end{theorem}

We aim to generalize the result in Theorem~\ref{thm:fsu93} for a general vertex $v$ as in the following theorem. 
\begin{theorem} \label{thm:main}
    Let $d = \gcd(m+1, n+1)$ and $v=(a,b)$. Then $(G_{m\times n}, v)$ is a P-position if and only if $d \nmid a$ and $d \nmid b$.
\end{theorem}
The proof comes directly from Theorems~\ref{thm:N-pos} and \ref{thm:P-pos}. We also note that Theorem~\ref{thm:fsu93} is a special case of Theorem~\ref{thm:main} when $v$ is a corner of $G_{m\times n}$.

An \emph{even kernel} for a simple graph $G$ is a nonempty vertex set $S$ such that
\begin{enumerate}
    \item $S$ is independent and
    \item for any vertex $u\notin S$, we have $|N(u)\cap S|$ is even.
\end{enumerate}
An even kernel is used in \cite{Fraenkel1993} as a main tool to determine a P-position.

\begin{theorem}[\cite{Fraenkel1993}]\label{EvenKernel}
    If $S$ is an even kernel for a simple graph $G$ and $v\in S$, then $(G,v)$ is a P-position.
    Moreover, if $G$ is a bipartite graph then $(G,v)$ is a P-position if and only if $v$ is in an even kernel for a graph $G$.
\end{theorem}
    
\section{Main results}

Let $v$ be a vertex of $G_{m\times n}$. We consider a \emph{ray} with slope $+1$ or $-1$ going out from $v$. There are four directions the ray can go, which can be denoted as NE, NW, SE and SW for the directions $(+1,+1)$, $(-1,+1)$, $(+1,-1)$ and $(-1,-1)$, respectively.
The ray continues until it hits the sides of the rectangle $[0, m+1]\times [0,n+1]$ on which it reflects in the usual manner. We assume that the ray will stop if it hits one of the four corners of the rectangle or it returns to $v$. In addition, we can also consider two rays going out from $v$ in different directions. As a result, we obtain a \emph{trail}, which is a collection of straight line segments among $v$ and vertices on the boundaries of the rectangle. Two examples of a trail at the vertex $v=(2,2)$ of $G_{5,3}$ are shown in Figure~\ref{fig:1}. Note that the first trail is obtained from a single ray going out from $v$ in the direction $(+1,+1)$ (or in the direction $(+1,-1)$), while the second trail is obtained from two rays going out from $v$ in the directions $(-1,+1)$ and $(-1,-1)$.

\begin{figure}[!ht]
    \centering
    \includegraphics[width=0.25\textwidth]{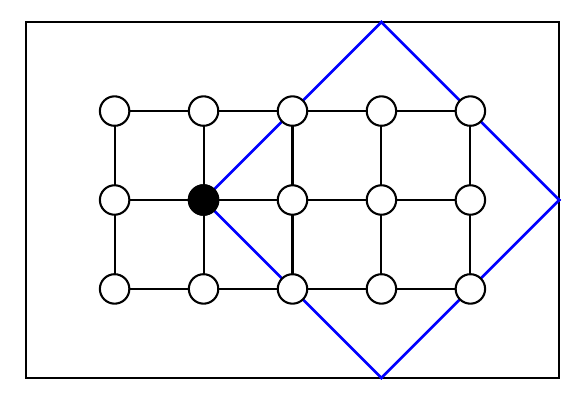}
    \hspace{0.1\textwidth}
    \includegraphics[width=0.25\textwidth]{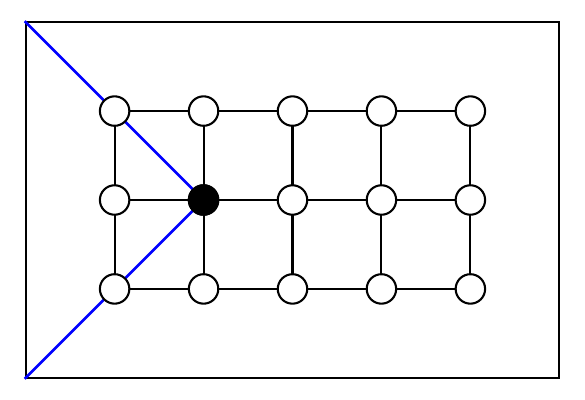}
    \caption{The vertex $(2,2)$ of $G_{5,3}$ has exactly two trails that there are two straight line segments touching $(2,2)$.}
    \label{fig:1}
\end{figure}

In this paper, we are only interested in a trail in which there are exactly two straight line segments touching $v$. We call it a \emph{90--degree trail} or a \emph{180--degree trail} depending on the angle between the two straight line segments at $v$. A trail is called \emph{open} if both ends go to corners. Otherwise, the trail is \emph{closed}. For the example in Figure~\ref{fig:1}, the vertex $(2,2)$ has exactly two 90--degree trials. The first 90--degree trail is closed while the second is open. The vertex $(2,2)$ has no 180--degree trail.

\subsection{N-position and 90--degree trail}

A 90--degree trail at $v$ separates the rectangle $[0,m+1] \times [0,n+1]$ into disjoint (polygonal) regions.
A region is \emph{unbounded} if at least one of its sides (boundary line segments) is a part of a side of the rectangle.
We give a label \emph{positive} or \emph{negative} to each region. A labeling is \emph{proper} if any two regions that share a side have different labels.

To prove Theorem~\ref{thm:N-pos}, we need Lemmas~\ref{lem:N1_closed_90_exist}, \ref{lem:N2_proper_label} and \ref{lem:N3_ek_90_trail}.

\begin{lemma}\label{lem:N1_closed_90_exist}
    Let $G= G_{m\times n}$ where $d = \gcd(m+1,n+1)$ and $u=(a,b)$ be a vertex of $G$.
    If $d \mid a$ or $d \mid b$, then there is a closed 90--degree trail at $u$. Furthermore, it can be obtained from a ray going from $u$ in the direction $(+1,+1)$ or $(-1,-1)$.
\end{lemma}
\begin{proof}
    Let $d \mid a$ or $d \mid b$. First, we assume that $d \mid a$. (The case $d \mid b$ can be treated similarly and gives the same result.) 
    We can write $a$ as a linear combination of $m+1$ and $n+1$ with integer coefficients as $a = \pm h(m+1) \mp k(n+1)$ where $h,k$ are positive integers.
    In addition, using $h'= hk(n+1)-h$ and $k'=hk(m+1)-k$, we also get $a = \mp h'(m+1) \pm k'(n+1)$.
    Without loss of generality, let us assume that
    \[ a = h(m+1) - k(n+1) \quad\text{and}\quad a = -h'(m+1)+k'(n+1), \]
    where $h,k,h'$ and $k'$ are positive integers.

    From $a=h(m+1)-k(n+1)$, we embed $2h(2k+1)$ copies (with $2h$ columns and $2k+1$ rows) of  $[0,m+1]\times[0,n+1]$ (which has $G_{m\times n}$ inside) on $[0,2h(m+1)]\times[0,(2k+1)(n+1)]$. 
    Any vertices $(i,j)$ and $(2p(m+1)\pm i,2q(n+1)\pm j)$ are from the same vertex in $[0,m+1]\times[0,n+1]$ where $p,q$ are integers. 
    In other words, the lines $x=\ell(m+1)$ for $1\leq\ell\leq 2h-1$ and the lines $y = \ell'(n+1)$ for $1\leq\ell'\leq 2k$ act as lines of reflection. 

    The equation $a = h(m+1) - k(n+1)$ can be rewritten as $2h(m+1)-2a = 2k(n+1)$. 
    This means the vertex $(a,b)$ and its copy at $(2h(m+1)-a, 2k(n+1)+b)$ can be joined by a straight line with slope $1$. 
    Since $(2h(m+1)-a,2k(n+1)+b)$ is a copy of $(a,b)$ on the $2h$-th column and $(2k+1)$-th row of $[0,m+1]\times [0,n+1]$, if we consider a non-stop ray from $(a,b)$ in $(+1,+1)$ direction (going to NE) that continues forever (does not stop at $(a,b)$ or at corners), it will return to $(a,b)$ (not necessary the first time) from SE. See Figure~\ref{fig:2}.

    For the case $a=-h'(m+1)+k'(n+1)$, it can be rewritten as $2h'(m+1)+2a = 2k'(n+1)$. We can embed $(2h'+2)$ columns and $(2k'+1)$ rows of $[0,m+1]\times[0,n+1]$ on $[-(2h'+1)(m+1), m+1 ]\times[-2k'(n+1), n+1]$. The vertices $(-2h'(m+1)-a,-2k'(n+1)+b)$ and $(a,b)$ can be joined by a straight line with slope $1$. So, there is a non-stop ray from $(a,b)$ in $(-1,-1)$ direction (going to SW) that returns to $(a,b)$ not necessary the first time) from NW.

    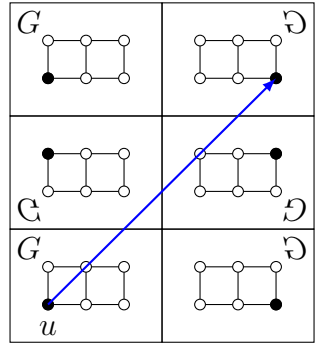
\begin{figure}[!ht]
    \centering
    \begin{tikzpicture}[scale=0.5]
        \begin{scope}
        \draw (0,0)--(4,0)--(4,3)--(0,3)--cycle; \node[] at (0.5,2.5) {$G$}; \draw[] (1,1)--(3,1) (1,2)--(3,2) (1,1)--(1,2) (2,1)--(2,2) (3,1)--(3,2); \foreach \i in {1,2,3} \foreach \j in {1,2} \node[circle, fill=white, draw, inner sep=0.05cm] at (\i,\j) {}; \node[circle, fill, draw, inner sep=0.05cm, label=below:{$u$}] at (1,1) {};
        \end{scope}
        \begin{scope}[xshift=8cm, xscale=-1]
        \draw (0,0)--(4,0)--(4,3)--(0,3)--cycle; \node[xscale=-1] at (0.5,2.5) {$G$}; \draw[] (1,1)--(3,1) (1,2)--(3,2) (1,1)--(1,2) (2,1)--(2,2) (3,1)--(3,2); \foreach \i in {1,2,3} \foreach \j in {1,2} \node[circle, fill=white, draw, inner sep=0.05cm] at (\i,\j) {};  \node[circle, fill, draw, inner sep=0.05cm] at (1,1) {};
        \end{scope}
        \begin{scope}[yshift=6cm, yscale=-1]
        \draw (0,0)--(4,0)--(4,3)--(0,3)--cycle; \node[yscale=-1] at (0.5,2.5) {$G$}; \draw[] (1,1)--(3,1) (1,2)--(3,2) (1,1)--(1,2) (2,1)--(2,2) (3,1)--(3,2); \foreach \i in {1,2,3} \foreach \j in {1,2} \node[circle, fill=white, draw, inner sep=0.05cm] at (\i,\j) {};  \node[circle, fill, draw, inner sep=0.05cm] at (1,1) {};
        \end{scope}
        \begin{scope}[xshift=8cm, yshift=6cm, yscale=-1, xscale=-1]
        \draw (0,0)--(4,0)--(4,3)--(0,3)--cycle; \node[xscale=-1, yscale=-1] at (0.5,2.5) {$G$}; \draw[] (1,1)--(3,1) (1,2)--(3,2) (1,1)--(1,2) (2,1)--(2,2) (3,1)--(3,2); \foreach \i in {1,2,3} \foreach \j in {1,2} \node[circle, fill=white, draw, inner sep=0.05cm] at (\i,\j) {};  \node[circle, fill, draw, inner sep=0.05cm] at (1,1) {};
        \end{scope}
        \begin{scope}[yshift=6cm]
        \draw (0,0)--(4,0)--(4,3)--(0,3)--cycle; \node[] at (0.5,2.5) {$G$}; \draw[] (1,1)--(3,1) (1,2)--(3,2) (1,1)--(1,2) (2,1)--(2,2) (3,1)--(3,2); \foreach \i in {1,2,3} \foreach \j in {1,2} \node[circle, fill=white, draw, inner sep=0.05cm] at (\i,\j) {};  \node[circle, fill, draw, inner sep=0.05cm] at (1,1) {};
        \end{scope}
        \begin{scope}[xshift=8cm, yshift=6cm, xscale=-1]
        \draw (0,0)--(4,0)--(4,3)--(0,3)--cycle; \node[xscale=-1] at (0.5,2.5) {$G$}; \draw[] (1,1)--(3,1) (1,2)--(3,2) (1,1)--(1,2) (2,1)--(2,2) (3,1)--(3,2); \foreach \i in {1,2,3} \foreach \j in {1,2} \node[circle, fill=white, draw, inner sep=0.05cm] at (\i,\j) {};  \node[circle, fill, draw, inner sep=0.05cm] at (1,1) {};
        \end{scope}
        \draw[-latex,thick,blue] (1,1)--(7,7);

        \begin{scope}[xshift=12cm]
        \draw (0,0)--(4,0)--(4,3)--(0,3)--cycle; \node[] at (0.5,2.5) {$G$}; \draw[] (1,1)--(3,1) (1,2)--(3,2) (1,1)--(1,2) (2,1)--(2,2) (3,1)--(3,2); \foreach \i in {1,2,3} \foreach \j in {1,2} \node[circle, fill=white, draw, inner sep=0.05cm] at (\i,\j) {}; \node[circle, fill, draw, inner sep=0.05cm, label=below:{$u$}] at (1,1) {};
        \draw[-latex,thick,blue] (1,1)--(3,3) (3,3)--(4,2) (4,2)--(2,0) (2,0)--(1,1);
        \end{scope}
    \end{tikzpicture}
    \caption{An embedding of multiple copies of $G_{3,2}$ on the rectangle $[0,8]\times[0,9]$ where $(h,k)=(1,1)$.}
    \label{fig:2}
    \end{figure}

    To finish the proof, it suffices to show that at least one of these two non-stop rays exiting from $(a,b)$ in directions $(+1,+1)$ and $(-1,-1)$ must return to $(a,b)$ for the first time from SE or NW without hitting a corner before returning. 
    This will conclude that there is a closed 90--degree trail at $(a,b)$.

    Assuming a contradiction, both non-stop rays exiting from $(a,b)$ in direction $(+1,+1)$ and $(-1,-1)$ return to $(a,b)$ for the first time from SW or NE, or hits a corner before returning. 
    
    We consider a non-stop ray exiting from $(a,b)$ in direction $(+1,+1)$. By the assumption, it returns for the first time from SW or NE, or hits a corner before returning. 
    Note that, by the first part of the proof, it must also return to $(a,b)$ (not necessary the first time) from SE.
    
    \begin{description}
        \item[Case 1] \textbf{It returns for the first time from SW.} The ray will leave to NE and return from SW, repeatedly. This is a contradiction to the fact that it must return (not necessary the first time) from SE.
        \item[Case 2] \textbf{It returns for the first time from NE.} The ray returns on the same path as it leaves $(a,b)$.
        % (So, it must have hit a corner.)
        The ray then continues its path and exits from $(a,b)$ in direction $(-1,-1)$. Now we think of this ray as a ray exiting from SW. From the assumption, its next return to $(a,b)$ must be from SW or NE.
        \begin{description}
            \item[Case 2.1] \textbf{The next return is from NE.} After returning, the ray will leave to SW and return from NE, repeatedly. This is a contradiction to the first part of the proof that it must return to $(a,b)$ (not necessary the first time) from SE.
            \item[Case 2.2] \textbf{The next return is from SW.} The ray then continues and exits from $(a,b)$ in direction $(+1,+1)$, and later returns from NE, leaves to SW, and returns from SW, repeatedly. This is a contradiction to the first part of the proof that it must return to $(a,b)$ (not necessary the first time) from SE.
        \end{description}
        % Hence, the next return after leaving to SW is from NW or SE.
        \item[Case 3] \textbf{It hits a corner before returning.} Since the ray exiting from $(a,b)$ in direction $(+1,+1)$, if it hits a corner before returning, then it must later returns to $(a,b)$ for the first time from NE. From Case 2, we get a contradiction.
    \end{description}
    All cases lead to a contradiction, hence at least one of these two non-stop rays exiting from $(a,b)$ in directions $(+1,+1)$ and $(-1,-1)$ must return to $(a,b)$ for the first time from SE or NW without hitting a corner before returning.
    So, there must be a closed 90--degree trail at $(a,b)$.
\end{proof}

\begin{lemma}\label{lem:N2_proper_label}
    The regions obtained from a closed 90--degree trail has a proper labeling that all unbounded regions are negative.
\end{lemma}
\begin{proof}
    Consider a closed 90--degree trail at a vertex $u$ of $G_{m\times n}$. We extend the two straight line segments touching $u$ so that the other end of each segment hits the boundary of the rectangle $[0,m+1]\times[0,n+1]$. It is easy to see that there is a proper labeling of the regions separated by this extended trail. We then switch the labels of the regions enclosed by the extended part of the two segments. After removing the extended part, we are left with a proper labeling. Since two consecutive unbounded regions share a common neighboring region, they have the same label. If the label of the unbounded regions is negative, we are done. Otherwise, we switch the label of every region and get the desired labeling. See Figure~\ref{fig:3} for an example.
\end{proof}
\begin{figure}[!ht]
    \centering
    \begin{subfigure}{0.45\textwidth}
        \includegraphics[width=\textwidth]{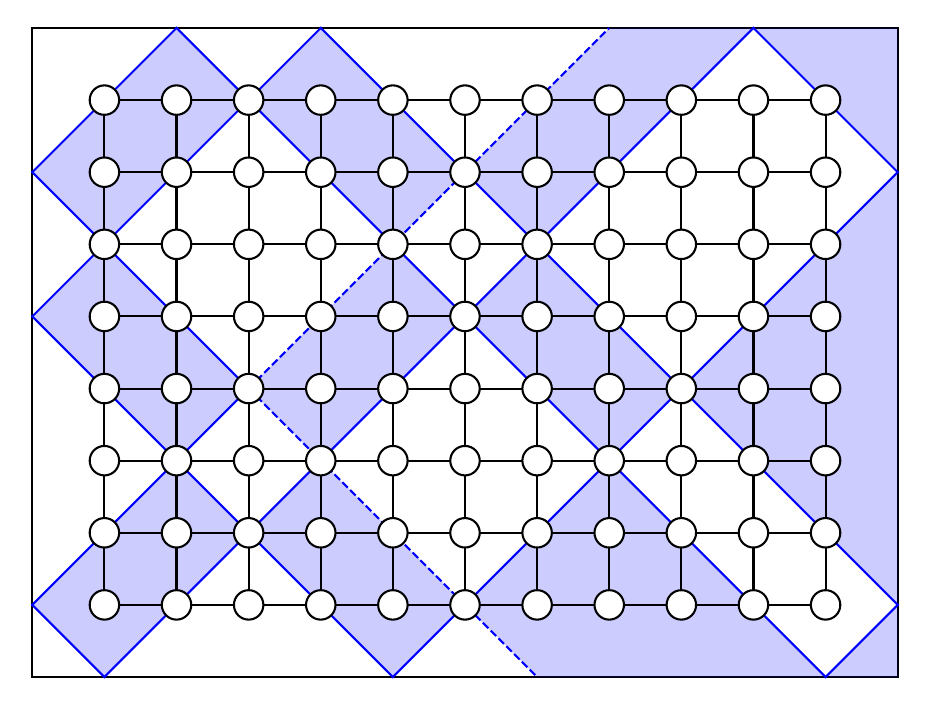}
        \caption{}
    \end{subfigure}
    \hspace{0.05\textwidth}
    \begin{subfigure}{0.45\textwidth}
        \includegraphics[width=\textwidth]{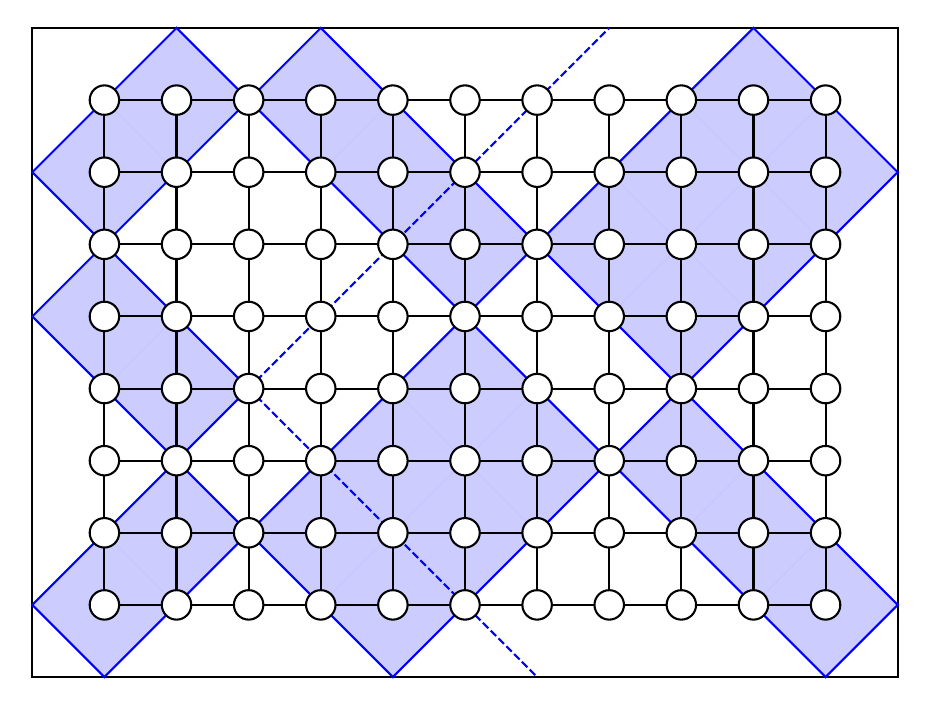}
        \caption{}
    \end{subfigure}
    \caption{(a) A proper labeling of the regions obtained from the extended trail at the vertex $(3,4)$ of $G_{11\times 8}$. (b) The result after switching the labels of the regions on the right of the extended segments. The positive regions are shown in blue.  All unbounded regions are negative.}
    \label{fig:3}
\end{figure}
    
\begin{lemma}\label{lem:N3_ek_90_trail}
    Let $\phi$ be a 90--degree trail at a vertex $u$ of the grid graph $G_{m\times n}$. 
    Let $\alpha$ be a proper labeling of the regions separated by a closed 90--degree trail a vertex $u=(a,b)$ of the grid graph $G_{m\times n}$ such that all unbounded regions are negative. Let $v$ be a neighbor of $u$ in a positive region.
    Then the set of vertices $(i,j)$ such that $i+j \not\equiv a+b \pmod{2}$ which are inside positive regions forms an even kernel of $G_{m\times n} - uv$ containing $v$.
\end{lemma}
\begin{proof}
    We say that the vertices $(i,j)$ and $(i',j')$ have the same parity if $i+j\equiv i'+j' \pmod{2}$.
    Let $K$ be the set of vertices of different parity from $u$ which are inside positive regions. We note that $v$ is in $K$ and has the same parity as every vertex in $K$. Also, $u$ has different parity to $v$ and has the same parity as every vertex on the trail.
    We now show that $K$ is an even kernel.
    First, we consider the vertex $u$. Since $u$ is adjacent in $G_{m\times n}$ to 1 or 3 vertices in $K$, it is adjacent in $G_{m\times n}-uv$ to 0 or 2 vertices in $K$. Let $w\neq u$ be a vertex of $G_{m\times n}$ that has the same parity as $u$. Consider the following cases.
    \begin{description}
        \item[Case 1] \textbf{$w$ is inside a positive region.} Then it has 4 neighbors in $K$.
        \item[Case 2] \textbf{$w$ is inside a negative region.} Then it has no neighbor in $K$.
        \item[Case 3] \textbf{$w$ is on the trail.} Then it has 2 neighbors in $K$.
    \end{description}
    From all cases, $w$ has even number of neighbors in $K$. Hence, $K$ is an even kernel of $G_{m\times n} - uv$ containing $v$.
\end{proof}

\begin{theorem} \label{thm:N-pos}
    Let $G= G_{m\times n}$ where $d = \gcd(m+1,n+1)$ and $u=(a,b)$ be a vertex of $G$. If $d \mid a$ or $d \mid b$, then $(G,u)$ is an N-position.
\end{theorem}
\begin{proof}
    From Lemma~\ref{lem:N1_closed_90_exist}, there is a closed 90--degree trail at $u$. We then get a proper labeling on the regions separated by the trail where all unbounded regions have the negative label. Let $v$ be a neighbor of $u$ in a positive region. Lemma~\ref{lem:N3_ek_90_trail} implies that a winning move from a vertex $u$ is to go to $v$. Hence, $(G,u)$ is an N-position.
\end{proof}

\subsection{P-position and 180--degree trail}

Lemmas~\ref{lem:P1_180_trail_exist} and \ref{lem:P2_ek_180_trail} are used to prove Theorem~\ref{thm:P-pos}.

\begin{lemma} \label{lem:P1_180_trail_exist}
    Let $G= G_{m\times n}$ where $d = \gcd(m+1,n+1)$ and $v=(a,b)$ be a vertex of $G$.
    If $d \nmid a$ and $d \nmid b$, then there is a (not necessary closed) 180--degree trail at $v$.
\end{lemma}
\begin{proof}
    Under the same construction as in the proof of Lemma~\ref{lem:N1_closed_90_exist}, we consider the tiling of $[0,m+1]\times[0,n+1]$ with $2(n+1)+1$ columns and $2(m+1)+1$ rows.
    It is clear that there is a non-stop ray leaving $(a,b)$ in the direction $(+1,+1)$ to $(2(n+1)(m+1)+a, 2(m+1)(n+1)+b)$. Hence, a non-stop ray leaving from $(a,b)$ in the direction $(+1,+1)$ must return to $(a,b)$ (not necessary the first time) from SW.

    Next, we will show that the ray must never return to $(a,b)$ from SE or NW (possibly hitting corners).
    Assume to the contrary that the ray returns to $(a,b)$ from SE or NW.
    Then there exist positive integers $h,k$ such that there is a straight line with slope $1$ joining $(a,b)$ to $(2h(m+1) - a,2k(n+1) + b)$ or $(2h(m+1)+a,2k(n+1) - b)$.
    Without loss of generality, we can assume that there is a straight line with slope $1$ joining $(a,b)$ and $(2h(m+1)-a,2k(n+1)+b)$.
    Thus
    \begin{equation*}
        \frac{(2k(n+1) + b)-b}{(2h(m+1)-a)-a} =1.
    \end{equation*}
    So $a=h(m+1)- k(n+1)$.
    Since $d=\gcd(m+1,n+1)$, we get $d\mid a$, a contradiction.
    Hence, the ray will never return to $(a,b)$ from SE or NW.
    Therefore, there is a (not necessary closed) 180--degree trail at $v$.
\end{proof}

\begin{lemma} \label{lem:P2_ek_180_trail}
    Let $v$ be a vertex of $G=G_{m\times n}$.
    The set of vertices on a 180--degree trail at $v$ that are on exactly one straight line segment forms an even kernel of $G$ containing $v$.
\end{lemma}
\begin{proof}
    Let $K$ be the set of vertices on a 180--degree trail at $v$ that are on exactly one straight line segment. It is clear that $v \in K$. Let $u$ be a vertex having a neighbor on the trail. We can run through all cases by the number of straight line segments whom $u$ is next to, see Figure~\ref{fig:4}. We see that $u$ can only have $0$, $2$ or $4$ neighbors in $K$. Hence, $K$ is an even kernel of $G$ containing $v$.
\end{proof}

\begin{figure}[!ht]
    \centering
    \resizebox{0.8\textwidth}{!}{
        \begin{tikzpicture}
        \tikzset{every node/.style={circle,thick,draw,fill=white,inner sep=0.12cm}}
            \draw[ultra thick, blue] (0,0) +(0.3,1.3) -- +(-1.3,-0.3); 
            \draw[ultra thick, blue] (4,0) +(0.3,1.3) -- +(-1.3,-0.3) +(-0.3,1.3) -- +(1.3,-0.3);
            \draw[ultra thick, blue] (8,0) +(0.3,1.3) -- +(-1.3,-0.3) +(-0.3,-1.3) -- +(1.3,0.3);
            \draw[ultra thick, blue] (12,0) +(0.3,1.3) -- +(-1.3,-0.3) +(-0.3,1.3) -- +(1.3,-0.3) +(-0.3,-1.3) -- +(1.3,0.3);
            \draw[ultra thick, blue] (16,0) +(0.3,1.3) -- +(-1.3,-0.3) +(-0.3,1.3) -- +(1.3,-0.3) +(-0.3,-1.3) -- +(1.3,0.3) +(-1.3,0.3) -- +(0.3,-1.3);
            \foreach \x in {0,4,8,12,16} {
                \draw (\x-1,0)--(\x+1,0) (\x,-1)--(\x,1);
                \node at (\x,0) {}; \node at (\x,1) {}; \node at (\x,-1) {}; \node at (\x-1,0) {}; \node at (\x+1,0) {};
            };
        \end{tikzpicture}
    }
    \caption{All possibilities of the neighbors on a trail.}
    \label{fig:4}
\end{figure}
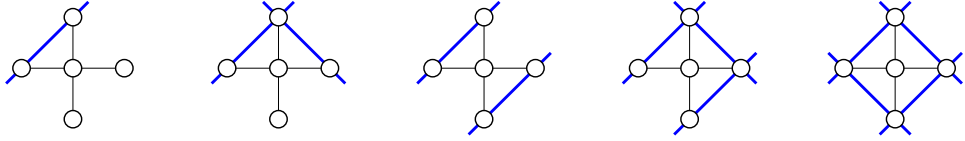

\begin{theorem} \label{thm:P-pos}
    Let $G= G_{m\times n}$ where $d = \gcd(m+1,n+1)$ and $v=(a,b)$ be a vertex of $G$.
    If $d \nmid a$ and $d \nmid b$, then $(G,v)$ is a P-position.
\end{theorem}
\begin{proof}
    Lemma~\ref{lem:P1_180_trail_exist} gives a 180--degree trail at $v$. Then, we can construct an even kernel of $G$ containing $v$ using Lemma~\ref{lem:P2_ek_180_trail}. Hence, $(G,v)$ is a P-position.
\end{proof}

\begin{example}
Consider $G=G_{11\times 8}$.
In Figure~\ref{fig:5}(a), an even kernel of $G$ containing a vertex $(2,4)$ is obtained from a closed 180--degree trail at $(2,4)$. So, $(G,(2,4))$ is a P-position.

On the other hand, at the vertex $u=(3,4)$, there is a winning move to $v=(2,4)$. This is confirmed by an even kernel of $G-uv$ containing $v$ obtained from a closed 90--degree trail at $u$ as shown in Figure~\ref{fig:5}(b).
\end{example}
\begin{figure}
    \centering
    \begin{subfigure}{0.45\textwidth}
        \includegraphics[width=\textwidth]{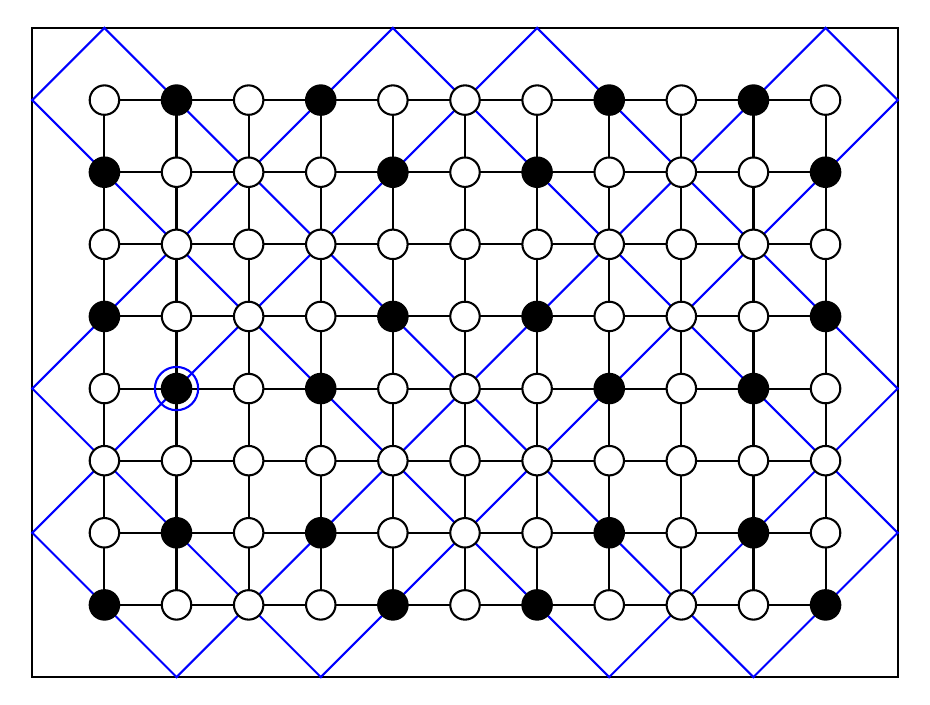}
        \caption{}
    \end{subfigure}
    \hspace{0.05\textwidth}
    \begin{subfigure}{0.45\textwidth}
        \includegraphics[width=\textwidth]{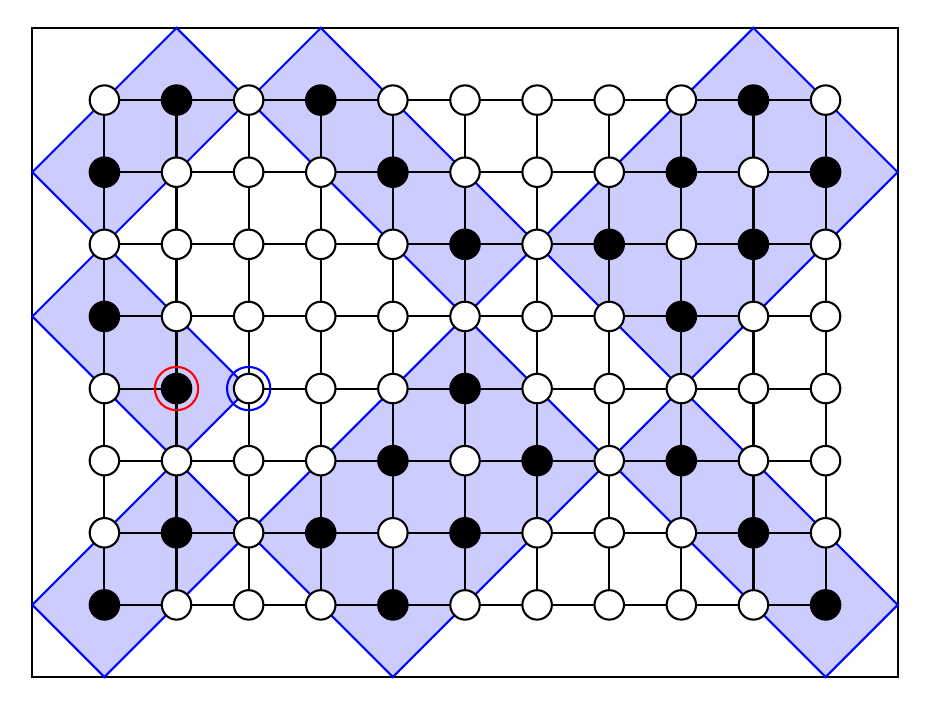}
        \caption{}
    \end{subfigure}
    \caption{(a) An even kernel (black vertices) of $G_{11\times 8}$ containing $(2,4)$ (blue circle). (b) An even kernel (black vertices) of $G_{11\times 8}-uv$ where $u=(3,4)$ (blue circle) and $v=(2,4)$ (red circle). The positive regions are shown in blue.}
    \label{fig:5}
\end{figure}

Note that the even kernel obtained from a 180--degree trail in Lemma~\ref{lem:P2_ek_180_trail} has a very nice pattern which can be explicitly computed. It has a repetitive pattern modulo $2d$ in both coordinates. In fact, it has a repetitive pattern modulo $d$ up to reflections.

\begin{corollary}
    Let $G= G_{m\times n}$ where $d = \gcd(m+1,n+1)$. Then the set
    \[
        S_k
        = \{(i,j) : d \nmid i \text{,  } d\nmid j \text{ and  } i\pm j \equiv \pm k  \pmod {2d}\}
    \]
    is an even kernel of $G$ for all $k=0,1,\dots,d$.
    Furthermore, a vertex $(a,b)$ where $d\nmid a$ and $d\nmid b$ is in $S_k$ for exactly two values of $k\in\{0,1,\dots,d\}$.
\end{corollary}

An example of the $S_k$ and its corresponding 180--degree trail for all $k=0,1,\dots,d$ when $G=G_{19,14}$ are shown in Figure~\ref{fig:6}.

\begin{figure}
    \centering
    \foreach \k in {0,1,2,3,4,5} {
        \begin{subfigure}{0.49\textwidth}
            \includegraphics[width=\textwidth]{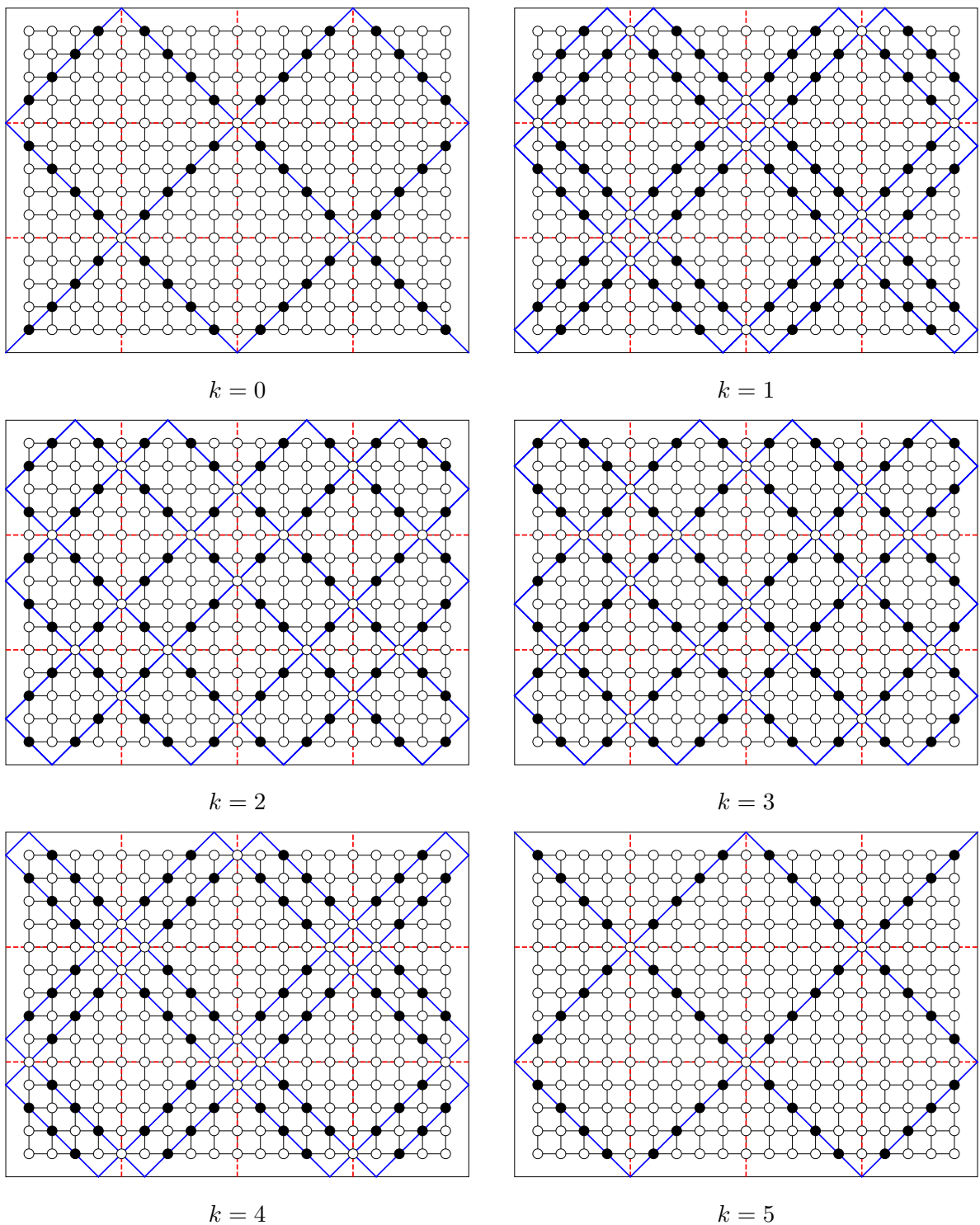}
            \caption*{$k=\k$}
        \end{subfigure}    
    }
    \caption{The even kernel $S_k$ of $G_{19,14}$ and its corresponding 180--degree trail for $k=0,1,\dots,5$.}
    \label{fig:6}
\end{figure}

\bibliographystyle{alpha} 
\bibliography{ref}
\end{document}